\documentclass[12pt,reqno]{amsart}

\usepackage{amsfonts,amsmath,amssymb,amsthm}
\usepackage{enumitem}
\usepackage{amsaddr}
\usepackage{mathrsfs}

\newtheorem{Theorem}{Theorem}[section]
\newtheorem{Lemma}{Lemma}[section]
\newtheorem{Proposition}{Proposition}[section]

\newtheorem{Remark}{Remark}[section]
\newtheorem{Definition}{Definition}[section]
		
\numberwithin{equation}{section}

\newcommand{\funcion}[5]{%
{\setlength{\arraycolsep}{2pt}
	\begin{array}{r@{}ccl}
	#1\colon & #2 & \longrightarrow & #3\\
	& #4 & \longmapsto & #5
\end{array}}}

\newcommand{\func}[3]{#1\colon  #2  \to  #3}

\newcommand{\Kp}{\mathscr{K}}

\DeclareMathOperator{\inte}{int}

\usepackage[a4paper,margin=3cm]{geometry}

\begin{document}

\title[compact countable subsets of a metric space]
{Countable ordinal spaces and compact countable  
subsets of a metric space}

\author[B. \'Alvarez-Samaniego]{Borys \'Alvarez-Samaniego}
\address{N\'ucleo de Investigadores Cient\'{\i}ficos\\
	Universidad Central del Ecuador (UCE)\\
	Quito, Ecuador}
\email{balvarez@uce.edu.ec, borys\_yamil@yahoo.com, balvarez@impa.br}

\author[A. Merino]{Andr\'es Merino}
\address{Escuela de Ciencias F\'isicas y Matem\'atica\\
    Facultad de Ciencias Exactas y Naturales\\
    Pontificia Universidad Cat\'olica del Ecuador (PUCE)\\
    Quito, Ecuador}
\email{aemerinot@puce.edu.ec}

\date{July 27, 2018.} 

\begin{abstract}
We show in detail that every compact countable subset of a metric space is 
homeomorphic to a countable ordinal number, which extends a result given by 
Mazurkiewicz and Sierpinski for finite-dimensional Euclidean spaces.  In order 
to achieve this goal, we use Transfinite Induction to construct a specific 
homeomorphism.  In addition, we prove that for all metric space $(E,d)$, the 
cardinality of the set of all the equivalence classes $\Kp_E$, up to homeomorphisms, 
of compact countable subsets of $E$ is less than or equal to $\aleph_1$, i.e. 
$|\Kp_E| \le \aleph_1$.  We also show that for all cardinal number $\kappa$ smaller 
than or equal to $\aleph_1$, there exists a metric space $(E_{\kappa}, d_{\kappa})$ 
such that $|\Kp_{E_{\kappa}}|= \kappa$.  
\end{abstract}


\maketitle

\section{Introduction} \label{sec:intro}

The study of homeomorphisms between compact countable subsets of a topological space 
and countable ordinal numbers was began by S. Mazurkiewicz and W. Sierpinski 
in~\cite{Sierpinski}. More precisely, they showed that for every compact countable 
subset of an $n$-dimensional Euclidean space, there exists a homeomorphism between this 
subset and some countable ordinal number. Moreover, a detailed proof of this last result 
when the Euclidean space, under consideration, is the real line, was given by the authors 
in~\cite{Alvarez-Merino2016}. Some related propositions can also be found 
in~\cite{Baker1973b, Baker1973a}.  The main result of Section~\ref{sec:CB} is 
Theorem~\ref{Theorem} below, which extends Lemma~3.6 in~\cite{Alvarez-Merino2016} for 
an arbitrary metric space.  It is worth mentioning that Theorem 2 in~\cite{Baker1972} 
considers compact, dispersed topological spaces with some additional properties, while 
Theorem~\ref{Theorem} below regards the case of a metric space.  Furthermore, it is 
stated in~\cite{Gillam2005}, without proof, that it is a known fact which can be proved 
by induction that  $Y$ is a countable locally compact space if and only if $Y$ is 
homeomorphic to some countable ordinal number (with the order topology).  In this way, 
Lemmas~\ref{L1}, \ref{L2} and~\ref{L3}, proved in Section~\ref{sec:CB}, are the 
comprehensive induction steps required in the Transfinite Induction used in the 
proof of Theorem~\ref{Theorem}. Section~\ref{sec:CP} is devoted to study the 
cardinality of the set of all the equivalence classes $\Kp_M$, up to homeomorphisms, of 
compact countable subsets of a metric space $(M, d)$.  Propositions~\ref{prop:3.01} and 
\ref{prop:3.02} are used in the proof of Theorem~\ref{teo:P4}, where it is shown that 
for all metric space $(E, d)$, the cardinality of $\Kp_E$ is less than or equal to 
$\aleph_1$.  Propositions~\ref{prop:3.03} to \ref{prop:3.05} shows that 
for all cardinal number $\kappa \le \aleph_1$, there exists a metric space 
$(E_{\kappa}, d_{\kappa})$ such that the cardinality of the set $\Kp_{E_{\kappa}}$ 
is equal to $\kappa$.  Proposition~\ref{prop:3.06} says that there exists a countable 
metric space $(F,d_F)$ such that $|\Kp_F|=\aleph_1$.  Finally, 
Proposition~\ref{prop:3.07} asserts that there is an uncountable metric 
space $(G,d_G)$ such that $|\Kp_G|=\aleph_0$.

We denote by ${\bf{OR}}$, the class of all ordinal numbers. In addition, $\omega$ 
represents the set of all natural numbers and $\omega_1$ is the set of all countable 
ordinal numbers. Further, we consider any ordinal number as being a topological space, 
endowed with its natural order topology.  In order to describe this last topology, for 
all $\alpha,\beta\in{\bf OR}$ such that $\alpha \le \beta$, we write 
\begin{align*}
	(\alpha,\beta) & := \{\gamma\in{\bf OR}: \alpha<\gamma<\beta \},\\
	[\alpha,\beta) & := \{\gamma\in{\bf OR}: \alpha\leq\gamma<\beta \}. 
\end{align*}
Thus, for any $\delta\in{\bf OR}$, the natural order topology for $\delta$ is given by the 
following topological basis
\begin{equation*}\label{eq:TO}
	\{(\beta,\gamma):\beta,\gamma\in{\bf OR},\ \beta<\gamma\leq \delta\}
	\cup\{[0,\beta):\beta\in{\bf OR},\ \beta\leq \delta\}.
\end{equation*}

Next definition was first introduced by G. Cantor in~\cite{Cantor1879}.  
\begin{Definition}[Cantor-Bendixson's derivative]\label{def:D_CB}
Let $A$ be a subset of a topological space. For a given ordinal number 
$\alpha \in \bf{OR}$, we define, using Transfinite Recursion, the 
\emph{$\alpha$-th derivative} of $A$, written $A^{(\alpha)}$, as follows:
\begin{itemize}
  \item $A^{(0)}=A$,
  \item $A^{(\beta+1)}=(A^{(\beta)})'$, for all ordinal number $\beta$,
  \item $\displaystyle A^{(\lambda)}=\bigcap_{\gamma<\lambda} A^{(\gamma)}$, 
        for all limit ordinal number $\lambda\neq 0$,
\end{itemize}
where $B'$ denotes the derived set of $B$, i.e., the set of all 
limit points (or accumulation points) of the subset $B$.
\end{Definition}
\begin{Remark} \label{R1}
Given any subset of a $T_1$ topological space, its derived set is closed. 
As a consequence of this last result, we have that if $F$ is a closed subset of a 
$T_1$ topological space, then $(F^{(\alpha)})_{\alpha\in{\bf OR}}$ is a decreasing 
family of closed subsets.  
\end{Remark}
%
Moreover, $\mathcal{P}(C)$ and $|C|$ denote, respectively, the power set and the 
cardinality of the set $C$. We also write $A \sim B$ if 
there is a homeomorphism between the topological spaces $A$ and $B$. If $(X,\tau)$ is a 
topological space, then $\mathcal{K}_X$ represents the set of all compact countable 
subsets of $X$, where a countable set is either a finite set or a countably infinite set, 
and $\Kp_X := \mathcal{K}_X /\!\sim$ denotes the set of all the equivalence classes, 
up to homeomorphisms, of elements of $\mathcal{K}_X$.  If $(E,d)$ is a metric 
space, $x \in E$ and $r>0$, we denote by $B(x,r)$ and $B[x,r]$ the open and closed 
balls, centered at $x$ with radius $r>0$, respectively. Furthermore, for all 
$Y \neq \varnothing$, $\rho_Y$ is used to designate the discrete metric on the 
set $Y$.  We now give the following definition.
\begin{Definition}[Cantor-Bendixson's characteristic]
Let $D$ be a subset of a topological space such that there exists an ordinal number 
$\beta \in \bf{OR}$ with the property that $D^{(\beta)}$ is finite.  We say that 
$(\alpha,p)\in {\bf OR} \times\omega$ is the \emph{Cantor-Bendixson characteristic} of 
$D$ if $\alpha$ is the smallest ordinal number such that $D^{(\alpha)}$ is finite 
and $|D^{(\alpha)}|=p$. In this case, we write $\mathcal{CB}(D)=(\alpha,p)$.
\end{Definition}
For the sake of completeness, we give here the proof of the following theorem, which 
was first introduced by G. Cantor in~\cite{Cantor1883s} for an $n$-dimensional 
Euclidean space.  It deserves to point out that there are some known extensions,  
considering topological spaces, of the next result.
\begin{Theorem} \label{T1}
Let $(X,\tau)$ be a Hausdorff space. For all $K \in \mathcal{K}_X$, there exists 
$\alpha \in \omega_1$ such that $K^{(\alpha)}$ is a finite set.  
\end{Theorem}
\begin{proof}
Let $K \in \mathcal{K}_X$.  We suppose, for a contradiction, that for all 
countable ordinal number $\gamma$, $K^{(\gamma)}$ is an infinite set.  Let 
$\alpha \in \omega_1$. Since $\alpha+1 \in \omega_1$, we have that 
$K^{(\alpha+1)}$ is an infinite set, and thus it is a nonempty set.  By 
Remark~\ref{R1}, $K^{(\alpha+1)} \subseteq K$.  Then, $K^{(\alpha+1)}$ is 
a countable set.  Using the fact that every nonempty, compact, perfect, 
Hausdorff space is uncountable, we obtain that 
$K^{(\alpha+2)}\neq K^{(\alpha+1)}$. Thus, by using again Remark~\ref{R1}, 
we get $K^{(\alpha+2)}\varsubsetneq K^{(\alpha+1)}$.  We now define 
\begin{equation*}
    K_\alpha:=K^{(\alpha+1)}\smallsetminus K^{(\alpha+2)}\neq\varnothing.
\end{equation*}
Then, $\{K_\gamma: \gamma \in \omega_1\}$ is a family of nonempty sets.  By 
the Axiom of Choice, there exists a function 
\begin{equation*}
    \func{f}{\omega_1}{\bigcup_{\gamma \in \omega_1} K_{\gamma}}
\end{equation*}
such that for all $\gamma \in \omega_1$, $f(\gamma) \in K_{\gamma}$. We 
claim that $f$ is injective.  In fact, let $\beta, \delta \in \omega_1$ be 
such that $\beta < \delta$.  Thus, $\beta+2\leq\delta+1$. By Remark~\ref{R1},
\begin{equation*}
    K^{(\delta+1)}\subseteq K^{(\beta+2)}.
\end{equation*}
Then, 
\begin{equation*}
    K_\beta\cap K_\delta=(K^{(\beta+1)}\smallsetminus K^{(\beta+2)})
    \cap (K^{(\delta+1)}\smallsetminus K^{(\delta+2)})=\varnothing.
\end{equation*}
Since $f(\beta) \in K_{\beta}$ and $f(\delta) \in K_{\delta}$, it 
follows that $f(\beta) \neq f(\delta)$.  Hence, $f$ is a one-to-one function. 
Therefore, 
\begin{equation*}
   \aleph_1 := |\omega_1| 
            \leq \left|\bigcup_{\gamma \in \omega_1} K_{\gamma}\right| 
            \leq |K|\leq \aleph_0,
\end{equation*}
giving a contradiction. This finishes the proof of the theorem.
\end{proof}
\begin{Remark} \label{R2}
Last theorem implies that if $(X,\tau)$ is a Hausdorff space and 
$K \in \mathcal{K}_X$, then $\mathcal{CB}(K)$ is well-defined and furthermore   
$\mathcal{CB}(K) \in \omega_1 \times \omega$.
\end{Remark}
%
\section{Existence of homeomorphisms} \label{sec:CB}

\begin{Lemma}\label{L1}
If $(E,d)$ is a metric space, $K\in \mathcal{K}_E$, and $\mathcal{CB}(K)=(1,1)$, then
\begin{equation*}
    K\sim \omega+1.
\end{equation*}
\end{Lemma}
\begin{proof}
Since $\mathcal{CB}(K)=(1,1)$, there exists $x\in E$ such that $K'=\{x\}$. Moreover, 
we see that $K=K^{(0)}$ is infinite. Then, $K\smallsetminus K'$ is a countably 
infinite set.  Thus, there is a bijection $g$ from $K\smallsetminus K'$ 
onto $\omega$. We now define the following function
\begin{equation*}
    \funcion{f}{K}{\omega+1}{z}{f(z)=\begin{cases}
    g(z), & \text{if }z\neq x,\\
    \omega, & \text{if }z= x.
    \end{cases}}
\end{equation*}
From the definition of $f$, we obtain directly that $f$ is a bijective function.  We 
will now show that $f$ is continuous. Since every point belonging to 
$K\smallsetminus K'$ is an isolated point of $K$, it follows that $f$ is continuous 
at every point of $K\smallsetminus K'$.  Thus, it remains to show the continuity 
of $f$ at the point $x$.  We take an open basic neighborhood $V$ of 
$f(x)=\omega$ with regard to the order topology of $\omega + 1$.  We will now show 
that $f^{-1}(V)$ is a neighborhood of $x$. If $V = [0, \beta)$, we have that 
$\beta = \omega +1$. Thus, $V= \omega+1$ and $f^{-1}(V)=K$ is a neighborhood of $x$.  
On the other hand, if $V = (n,\alpha)$, then 
\begin{equation*}
    n <\omega <\alpha \leq \omega+1.
\end{equation*}
Therefore, $n\in\omega$ and $\alpha =\omega+1$. We now define the following set
\begin{equation*}
    A :=\{z\in K: f(z)\leq n\}.
\end{equation*}
Thus, $x\not\in A$.  Moreover, since $f$ is an injective function, we see that 
$A$ is a finite set. Let us take $r:=\min\{d(z,x):z\in A\} >0$.  Then,
\begin{equation*}
    K\cap B(x,r) \subseteq f^{-1}((n,\omega +1)).
\end{equation*}
In fact, if $z\in K$ satisfies $d(z,x)<r$, then $z\not\in A$. Hence, $f(z)>n$.  In 
addition, using the definition of function $f$, we see directly that 
$f(z) < \omega + 1$.  Thus, $f(z) \in (n, \omega + 1)=V$.  Consequently, 
$f^{-1}((n,\omega +1))$ is a neighborhood of $x$. Therefore, $f$ is continuous 
at the point $x$.  We thus conclude that $f$ is continuous at every point of its 
domain.  Then, $f$ is a continuous function.  Finally, since $f$ is a continuous 
bijective function, $K$ is compact and $\omega+1$ is a Hausdorff space, it follows 
that $f$ is a homeomorphism. In conclusion, $K \sim \omega+1$.
\end{proof}

The next lemma extends Lemma~3.4 in~\cite{Alvarez-Merino2016} to the case of an 
arbitrary $T_1$ topological space.

\begin{Lemma}\label{Lemtech}
Let $K$ and $F$ be closed subsets of a $T_1$ topological space such that 
$K \cap F$ = $K \cap \inte(F)$, where $\inte(F)$ is the set of all interior points 
of $F$.  Then, for all $\alpha \in \mathbf{OR}$, we have that 
\begin{equation}\label{eq:lemtech}
    (K \cap F)^{(\alpha)} = K^{(\alpha)} \cap F.
\end{equation} 
\end{Lemma}
\begin{proof}
We  will use Transfinite Induction.  
\begin{itemize}[leftmargin=*]
\item The case $\alpha = 0$ follows directly. 
\item We assume that the result holds for a given $\alpha \in \mathbf{OR}$, 
      i.e.,  $(K \cap F)^{(\alpha)} = K^{(\alpha)} \cap F$.  Then, 
      \begin{equation*}
        (K \cap F)^{(\alpha +1)} = \left( (K \cap F)^{(\alpha)} \right)'
        = ( K^{(\alpha)} \cap F )'
        \subseteq ( K^{(\alpha)} )' \cap F'
        \subseteq K^{(\alpha +1)} \cap F,
      \end{equation*}
      where in the last expression we have used the fact that $F$ is closed.  To show  
      the other inclusion, let $x \in K^{(\alpha +1)} \cap F$.  Since $K$ is a 
      closed subset of a $T_1$ topological space, using Remark~\ref{R1}, it follows 
      that $x \in K \cap F = K \cap \inte(F)$.  Therefore, there exists a 
      neighborhood $U$ of $x$ such that $U \subseteq F$.  Let $V$ be a neighborhood 
      of $x$. We now take $W:=U\cap V$.  We see that $W$ is also a neighborhood of $x$. 
      Then,
      \begin{align*}
        \varnothing & \neq \big( W \smallsetminus \{ x\} \big) \cap K^{(\alpha)}\\
        & =\big( W \smallsetminus \{ x\} \big) \cap K^{(\alpha)} \cap F \\
        & = \big( W \smallsetminus \{ x\} \big) \cap (K \cap F)^{(\alpha)}\\
        & \subseteq \big( V \smallsetminus \{ x\} \big) \cap (K \cap F)^{(\alpha)}.
      \end{align*}
      Hence, $x \in (K \cap F)^{(\alpha +1)}$.  Thus, 
      $K^{(\alpha +1)} \cap F \subseteq (K \cap F)^{(\alpha +1)}$. Therefore, 
      $(K \cap F)^{(\alpha +1)} = K^{(\alpha +1)} \cap F$.
\item Lastly, let $\lambda \neq 0$ be a limit ordinal number.  We assume that for 
      all $\beta \in \mathbf{OR}$ such that $\beta < \lambda$, 
      $(K \cap F)^{(\beta)} = K^{(\beta)} \cap F$.  Hence, 
      \begin{equation*}
        (K \cap F)^{(\lambda)} = \bigcap_{\beta < \lambda} (K \cap F)^{(\beta)}
        = \bigcap_{\beta < \lambda} (K^{(\beta)} \cap F)
        = \bigcap_{\beta < \lambda} K^{(\beta)} \cap F 
        = K^{(\lambda)} \cap F.
      \end{equation*}
\end{itemize}
This finishes the proof.
\end{proof}

\begin{Lemma}\label{L2}
Let $(E,d)$ be a metric space and let $\alpha>1$ be a countable ordinal number.
Suppose that for all ordinal number $\beta \in \bf{OR}$ such that $0<\beta<\alpha$ 
and for all $\widetilde{K}\in\mathcal{K}_E$ with  
$\mathcal{CB}(\widetilde{K})=(\beta,p) \in {\bf OR} \times \omega$, we have 
that $\widetilde{K}\sim \omega^\beta\cdot p+1$.  Then, for all $K\in\mathcal{K}_E$ 
such that $\mathcal{CB}(K)=(\alpha,1)$, we get  
\begin{equation*}
    K\sim \omega^\alpha+1.
\end{equation*}
\end{Lemma}
\begin{proof}
Let $K\in\mathcal{K}_E$ be such that $\mathcal{CB}(K)=(\alpha,1)$ with $\alpha>1$.  Then, 
there exists $x\in K$ with $K^{(\alpha)}=\{x\}$. We see that 
$x\in K^{(\alpha)}\subseteq K''$. Thus, $x$ is an accumulation point of $K'$.  Then, 
there is a sequence $(x_n)_{n\in\omega}$ in $K'\smallsetminus\{x\}$ such that 
$(d(x_n,x))_{n\in\omega}$ is a strictly decreasing sequence converging to $0$. Moreover, 
since $\{d(z,x)\in\mathbb{R}: z\in K \}$ is a countable set, it follows that for 
all $n\in\omega$, 
\begin{equation*}
    A_n := \{d(z,x)\in\mathbb{R}: z\in K \}^c \cap (d(x_{n+1},x),d(x_n,x))
\end{equation*}
is a nonempty set.  Therefore, $\{A_n : n\in\omega \}$ is a nonempty family of 
nonempty sets. By the Axiom of Countable Choice, there exists a sequence 
$(r_n)_{n\in\omega}$ of real numbers such that for all $n\in\omega$,
\begin{equation*}
     d(x_{n+1},x) < r_n < d(x_n,x)
\end{equation*}
and for all $z\in K$ we have that $d(z,x)\neq r_n$. Thus, for all $n\in\omega$, we 
define the following sets
\begin{align*}
    F_0 & := B(x,r_0)^c,\\
    F_{n+1} & := B[x,r_n]\smallsetminus B(x,r_{n+1})
\end{align*}
and
\begin{equation*}
    K_n := K \cap F_n.
\end{equation*}
We claim that for all $n\in\omega$, 
\begin{equation*}
    K \cap F_n = K \cap \inte(F_n).
\end{equation*}
In fact, let $n \in \omega$.  We see immediately that 
$K \cap \inte(F_n) \subseteq K \cap F_n$. Reciprocally, let $z\in K \cap F_n$. We 
first consider the case when $n=0$.  We obtain that $z\in K$ and $d(z,x) \geq r_{0}$. 
Since $z\in K$, we have that $d(z,x)\neq r_0$. Thus,   
$\varepsilon_0 := d(z,x)-r_0>0$. It is not difficult to see now that 
$B(z,\varepsilon_0)\subseteq F_0$.   Then, $z\in\inte(F_0)$.  Hence, 
$K \cap F_0 \subseteq K \cap \inte(F_0)$.  We now consider the case 
$n \in \omega \smallsetminus\{0\}$.  We have that $z\in K$ and 
\begin{equation*}
    r_{n} \leq d(z,x) \leq r_{n-1}. 
\end{equation*}
Since $z\in K$, we obtain that $d(z,x)\neq r_{n}$ and $d(z,x)\neq r_{n-1}$. 
We now take $\varepsilon_n:= \min\{ d(z,x)-r_{n} , r_{n-1}-d(z,x) \}>0$. We get  
$B(z,\varepsilon_n)\subseteq F_n$. Hence, $z\in\inte(F_n)$. Therefore, 
$K \cap F_n \subseteq K \cap \inte(F_n)$. \\
We can now see that the family  $\{K_n : n\in\omega\}$ has the following 
properties:
\begin{itemize}
\item Since $K$ is a closed subset of $E$, we obtain that for all 
      $n\in \omega$, $x_n\in K_n$.
\item For all $n\in \omega$, $K_n\subseteq K$.
\item Since the intersection of two closed subsets is also closed, we see that 
      for all $n\in \omega$, $K_n$ is a closed subset. 
\item Since every closed subset of a compact space is compact, we have that for all 
      $n\in \omega$, $K_n$ is compact. 
\item Since for all $n\in \omega$, $K_n$ is a countable set, we obtain that 
      for all $n\in \omega$, $K_n\in \mathcal{K}_E$.
\item For all $n\in \omega$, $K_n'\neq \varnothing$. In fact, let $n\in\omega$. 
      By Lemma~\ref{Lemtech}, we have that $K_n'=(K\cap F_n)'=K'\cap F_n$.  Moreover,
      since $x_n\in K'\cap F_n$, we see that $x_n\in K_n'$.
\item Since $\{F_n : n\in\omega\}$ is a pairwise disjoint family of sets, we obtain 
      that the family of sets $\{K_n : n\in\omega\}$ is also pairwise disjoint.
\item We have that 
      \begin{equation*}
          K=\biguplus_{n\in\omega} K_n \uplus \{x\}.
      \end{equation*}
      In fact, since the sequence $(r_n)_{n\in\omega}$ converges to $0$, we see that  
      $\biguplus_{n\in\omega} F_n \uplus \{x\} = E$.  Then, 
      \begin{align*}
        \biguplus_{n\in\omega} K_n \uplus \{x\} 
         &= \biguplus_{n\in\omega} (K\cap F_n) \uplus \{x\}\\
         &= \left(K\cap\biguplus_{n\in\omega} F_n\right) \uplus \{x\}\\
         &= K\cap\left(\biguplus_{n\in\omega} F_n\uplus \{x\}\right) \\
         &= K \cap E = K.
      \end{align*}
\item For all $n\in \omega$, $K_n^{(\alpha)} = \varnothing$.  In fact, by 
      Lemma~\ref{Lemtech}, we see that for all $n\in\omega$, 
      \begin{equation*}
        K_n^{(\alpha)}
        = (K\cap F_n)^{(\alpha)}
        = K^{(\alpha)}\cap F_n 
        = \{x\}\cap F_n =\varnothing.
      \end{equation*}
\item Using the fact that an infinite subset of a compact subset of a
      topological space has at least a limit point in the compact subset, 
      using also Remark \ref{R1} and the Cantor intersection theorem in a 
      Hausdorff topological space, we see that the last assertion implies that  
      for all $n\in \omega$, 
      if $\mathcal{CB}(K_n)=(\beta_n,p_n) \in {\bf OR} \times \omega$, then 
      $0<\beta_n<\alpha$ and $p_n \in \omega \smallsetminus\{0\}$.
\end{itemize}
It follows from the hypothesis that for all $n\in \omega$, 
$K_n \sim \omega^{\beta_n}\cdot {p_n}+1$. By the Axiom of Countable Choice, 
there is a sequence $(f_n)_{n \in \omega}$ of homeomorphisms, such that for 
all $n\in\omega$, $\func{f_n}{K_n}{\omega^{\beta_n}\cdot {p_n}+1}$ is a 
homeomorphism of the topological space $K_n$ onto  
$\omega^{\beta_n}\cdot {p_n}+1$. We now define the following function
\begin{equation*}
\funcion{f}{K}{\tau+1}{z}
	{f(z)=\begin{cases}
	    f_0(z),
	        & \text{if }z\in K_0,\\[2mm]\displaystyle
		\sum_{k=0}^{n-1}\omega^{\beta_k}\cdot p_k+1+f_n(z), 
		    & \text{if }z\in K_n, \text { for some } n\in\omega\smallsetminus\{0\}, 
		    \\[2mm]
		\tau,
		    & \text{if }z=x,
		\end{cases}}
\end{equation*}
where 
\begin{equation*}
    \tau := \sum_{k\in\omega}\omega^{\beta_k}\cdot p_k
	:= \sup\left\{\sum_{k=0}^n\omega^{\beta_k}\cdot p_k:n  \in\omega\right\}.
\end{equation*}
Proceeding in a similar way to the proof of Lemma 3.3 in~\cite{Alvarez-Merino2016}, 
it is possible to show that $\tau=\omega^\alpha$ and $f$ is a homeomorphism 
of $K$ onto $\tau+1$.  Hence, $K\sim \omega^\alpha+1$.
\end{proof}

\begin{Lemma}\label{L3}
Let $(E,d)$ be a metric space.  Let $\alpha$ be a countable ordinal number 
such that $\alpha > 0$ and let $p \in \omega$. We assume that for all 
$\widetilde{K} \in \mathcal{K}_E$ such that $\mathcal{CB}(\widetilde{K})=(\alpha,1)$, 
we have that $\widetilde{K} \sim \omega^\alpha+1$.  Then, for all $K\in\mathcal{K}_E$ 
with $\mathcal{CB}(K)=(\alpha,p)$, we get  
\begin{equation*}
    K\sim \omega^\alpha\cdot p + 1.
\end{equation*}
\end{Lemma}
\begin{proof}
Let $K\in\mathcal{K}_E$ be such that 
$\mathcal{CB}(K)=(\alpha,p)$ with $\alpha>0$.  As mentioned in the proof of the 
previous lemma, we can show that $p \in \omega \smallsetminus\{0\}$. Thus, 
\begin{equation*}
    K^{(\alpha)}=\{x_0,x_1,\ldots,x_{p-1}\},
\end{equation*}
where for all $i, j \in \{0, \ldots, p-1\}$ such that $i \neq j$, $x_i\neq x_j$. 
For all $m \in \{1, \ldots, p-1\}$, we define 
\begin{equation*}
    d_m: = \min\{d(x_m,x_j)\in\mathbb{R}: 0\leq j \leq p-1  
    \quad\text{and}\quad j\neq m\} > 0.
\end{equation*}
Let $m \in \{1, \ldots, p-1\}$.  Since
\begin{equation*}
    \{d(z,x_m)\in\mathbb{R}: z\in K \}
\end{equation*}
is a countable set, there exists $r_m>0$ such that
\begin{equation*}
    r_m\in\{d(z,x_m)\in\mathbb{R}: z\in K \}^c \cap (0,d_m).
\end{equation*}
Thus, for all $z\in K$, we have that $d(z,x_m)\neq r_m$. We now define
\begin{equation*}
    F_m := B[x_m,r_m]
\end{equation*}
and 
\begin{equation*}
    F_0 := E\smallsetminus \bigcup_{j=1}^{p-1} B(x_j,r_j).
\end{equation*}
Moreover, for all $n \in \{0, \ldots, p-1\}$, we also define
\begin{equation*}
    K_n := K \cap F_n.
\end{equation*}
We observe that for all $n \in \{0, \ldots, p-1\}$, 
\begin{equation*}
    K \cap F_n = K \cap \inte(F_n).
\end{equation*}
In fact, let $n \in \{0, \ldots, p-1\}$.  Since $\inte(F_n) \subseteq F_n$, 
we see that $K \cap \inte(F_n) \subseteq K \cap F_n$. Reciprocally, given 
$z\in K \cap F_n$, we see that $z\in K$ and we consider the following two cases:
\begin{itemize}
\item We first examine the situation when $n \in \{1, \ldots, p-1\}$. Since 
      $F_n:=B[x_n,r_n]$, we have that 
      \begin{equation*}
         d(z,x_n) \leq r_{n}.
      \end{equation*}
      In addition,  $z\in K$ implies that $d(z,x_n) \neq r_{n}$.  By taking  
      $\varepsilon_n := r_{n}-d(z,x_n)>0$, we obtain that 
      $B(z,\varepsilon_n)\subseteq F_n$. Hence, $z\in\inte(F_n)$.
    
\item We now assume that $n= 0$. Using the fact that   
      $z \in F_0:=E\smallsetminus \bigcup\limits_{j=1}^{p-1} B(x_j,r_j)$, we 
      conclude that for all $j \in \{1, \ldots, p-1 \}$, 
      \begin{equation*}
          d(z,x_j) \geq r_{j}.
      \end{equation*}
      Furthermore, since $z\in K$, we see that for all $j \in \{1, \ldots, p-1 \}$, 
      $d(z,x_j)\neq r_{j}$. We now take 
      $\varepsilon_0:=\min\{ d(z,x_j)-r_{j}: 1\leq j\leq p-1\}>0$.  
      Then, $B(z,\varepsilon_0)\subseteq F_0$. In order to prove this last assertion, 
      let $w \in B(z, \varepsilon_0)$.  We now suppose, to derive a contradiction, that 
      there exists $i \in \{1, \ldots, p-1 \}$ such that $d(w,x_i) < r_i$.  Then, 
      $\varepsilon_0 + r_i \leq d(z,x_i) \leq d(z,w) + d(w,x_i) < \varepsilon_0 + r_i$, 
      which is a contradiction.  Thus, $z\in\inte(F_0)$.
\end{itemize}
Therefore, $K \cap F_n \subseteq K \cap \inte(F_n)$. \\
Proceeding in a similar manner as in the proof of Lemma~\ref{L2}, we can show that 
the family $\{K_n: 0 \leq n \leq p-1 \}$ satisfies the next properties:
\begin{itemize}
\item Using Remark \ref{R1}, we have that for all $n \in \{0, \ldots, p-1\}$, 
      $x_n\in K_n$.
\item For all $n \in \{0, \ldots, p-1\}$, $K_n \subseteq K$.
\item For all $n \in \{0, \ldots, p-1\}$, $K_n$ is a closed subset of $E$.
\item Since every closed subset of a compact space is compact, we obtain 
      that for all $n \in \{0, \ldots, p-1\}$, $K_n$ is compact.
\item For all $n \in \{0, \ldots, p-1\}$, $K_n\in \mathcal{K}_E$.
\item $\{K_n: 0 \leq n \leq p-1\}$ is a pairwise disjoint family of sets.
\item $\displaystyle K=\biguplus_{n=0}^{p-1} K_n$.
\item Using Lemma~\ref{Lemtech}, we conclude that for all $n \in \{0, \ldots, p-1\}$, 
      \begin{equation*}
          K_n^{(\alpha)} = (K\cap F_n)^{(\alpha)}
          = K^{(\alpha)}\cap F_n =\{x_n\}.
      \end{equation*}
\item It follows from the last assertion that for all 
      $n \in \{0, \ldots, p-1\}$, $\mathcal{CB}(K_n)=(\alpha,1)$.
\end{itemize}
By using the hypothesis, we see that for all $n  \in \{0, \ldots, p-1\}$, there exists a 
homeomorphism $\func{g_n}{K_n}{\omega^{\alpha}+1}$ from the topological space 
$K_n$ onto $\omega^{\alpha} + 1$. We now consider the following function
\begin{equation*}
\funcion{g}{K}{\tau+1}{z}
    {g(z)= \begin{cases} 
		g_0(z), 
		    & \text{if } z \in K_0, \\[2mm]\displaystyle
		\omega^{\alpha}\cdot n +1+g_n(z), 
		    & \text{if } z\in K_n,\text { for some }  n \in \{1, \ldots, p-1\},
		\end{cases}}
\end{equation*}
where $\tau:=\omega^{\alpha}\cdot p$.
By a similar argument to the one used in the proof of Lemma \ref{L2} above, we 
obtain that function $g$ is a homeomorphism from $K$ onto $\tau + 1$.  Hence,
$K\sim \omega^\alpha\cdot p+1$.
\end{proof}

\begin{Theorem}\label{Theorem}
Suppose that $(E,d)$ is a metric space. Let $\alpha$ be a countable ordinal number 
such that $\alpha > 0$ and let $p \in \omega$. If $K\in\mathcal{K}_E$ satisfies 
$\mathcal{CB}(K)=(\alpha,p)$, then 
\begin{equation*}
    K\sim \omega^\alpha\cdot p +1.
\end{equation*}
\end{Theorem}
\begin{proof}
We proceed by Strong Transfinite Induction on the ordinal number $\alpha > 0$.  By 
Lemmas~\ref{L1} and~\ref{L3}, the result holds for $\alpha=1$. Now, let 
$\alpha\in \omega_1$ be such that $\alpha>1$.  We suppose that the conclusion is 
true for all ordinal number $\beta$ such that $0<\beta<\alpha$. By Lemmas~\ref{L2} 
and ~\ref{L3}, the result is also valid for $\alpha$. Thus,  
the theorem holds for all countable ordinal numbers greater than zero.
\end{proof}

\begin{Remark}
The hypothesis about the countable cardinality of the ordinal number $\alpha$ in 
Lemma~\ref{L2}, Lemma~\ref{L3} and Theorem~\ref{Theorem} can be omitted. In fact, 
if $(E,d)$ is a metric space, $K\in\mathcal{K}_E$,  
$(\alpha,p) \in {\bf{OR}} \times (\omega \smallsetminus \{0\})$ and 
$K\sim \omega^\alpha\cdot p +1$, then $\alpha \in \omega_1$.
\end{Remark}

\section{Cardinality of the partition} \label{sec:CP}
Let $(X,\tau)$ be a topological space.  We consider the set $\mathcal{K}_X$ of all compact 
countable subsets of $X$.  The set $\Kp_X := \mathcal{K}_X /\!\sim$ provides a 
partition of the set $\mathcal{K}_X$ into disjoint equivalence classes, more precisely, 
\begin{equation*}  \label{def:Kp}
    \Kp_X=\{[K]\in\mathcal{P}(\mathcal{K}_X):K\in\mathcal{K}_X\},  
\end{equation*}
where, for all $K\in\mathcal{K}_X$ 
\begin{equation*}
    [K]:=\{K_1\in\mathcal{K}_X:K_1\sim K\}.
\end{equation*}
The following two propositions will be used in the proof of Theorem~\ref{teo:P4} below.
\begin{Proposition}\label{prop:3.01}
Let $(X,\tau)$ be a $T_1$ topological space.  For all $K_1,K_2\in \mathcal{K}_X$ such that 
$K_1\sim K_2$, we have that $\mathcal{CB}(K_1)=\mathcal{CB}(K_2)$.
\end{Proposition}
\begin{proof}
Let $K_1,K_2\in \mathcal{K}_X$ be such that $K_1\sim K_2$ and let $\func{f}{K_1}{K_2}$ 
be a homeomorphism from $K_1$ onto $K_2$. We will first show that for all ordinal 
number $\alpha \in \mathbf{OR}$, $K_1^{(\alpha)}\sim K_2^{(\alpha)}$, where 
$f|_{K_1^{(\alpha)}}$ is a homeomorphism between these two sets. In order to prove 
this last assertion, we use below Transfinite Induction.
\begin{itemize}[leftmargin=*]
\item In the case when $\alpha=0$, we see that $K_1^{(0)}=K_1\sim K_2=K_2^{(0)}$ and
      $\func{f=f|_{K_1^{(0)}}}{K_1^{(0)}}{K_2^{(0)}}$ is a homeomorphism from 
      $K_1^{(0)}$ onto $K_2^{(0)}$.
\item We now suppose that the result holds for a given ordinal number $\alpha$, i.e., 
      $K_1^{(\alpha)}\sim K_2^{(\alpha)}$ and $f|_{K_1^{(\alpha)}}$ is a homeomorphism
      between $K_1^{(\alpha)}$ and $K_2^{(\alpha)}$. By Remark~\ref{R1} above, we have 
      that $K_1^{(\alpha+1)}\subseteq K_1^{(\alpha)}$. Thus, 
      \begin{equation*}
    	f(K_1^{(\alpha+1)})=f|_{K_1^{(\alpha)}}(K_1^{(\alpha+1)})
    	=f|_{K_1^{(\alpha)}}((K_1^{(\alpha)})')=(K_2^{(\alpha)})' 
    	= K_2^{(\alpha+1)}.
      \end{equation*}
      Then, $\func{f|_{K_1^{(\alpha+1)}}}{K_1^{(\alpha+1)}}{K_2^{(\alpha+1)}}$ is a 
      homeomorphism from $K_1^{(\alpha+1)}$ onto $K_2^{(\alpha+1)}$.  Therefore, 
      $K_1^{(\alpha+1)} \sim K_2^{(\alpha+1)}$.
\item Finally, let $\lambda\neq 0$ be a limit ordinal number.  We presume that 
      for all $\beta \in \mathbf{OR}$ such that $\beta < \lambda$, 
      $K_1^{(\beta)}\sim K_2^{(\beta)}$, where $f|_{K_1^{(\beta)}}$ is a homeomorphism 
      from $K_1^{(\beta)}$ onto $K_2^{(\beta)}$. Since $f$ is an injection, we have that
      \begin{equation*}
    	f(K_1^{(\lambda)}) 
	    = f\left(\bigcap_{\beta<\lambda} K_1^{(\beta)} \right)
    	= \bigcap_{\beta<\lambda} f(K_1^{(\beta)})
    	= \bigcap_{\beta<\lambda} K_2^{(\beta)}
    	= K_2^{(\lambda)}.
      \end{equation*}
      Therefore, $f|_{K_1^{(\lambda)}}$ is a homeomorphism between $K_1^{(\lambda)}$ 
      and $K_2^{(\lambda)}$, i.e., $K_1^{(\lambda)}\sim K_2^{(\lambda)}$.
\end{itemize}
Then, for all $\alpha \in \mathbf{OR}$, $|K_1^{(\alpha)}|=|K_2^{(\alpha)}|$. We 
suppose that $\mathcal{CB}(K_1)=(\beta,p)\in {\bf OR} \times\omega$.  Thus, $\beta$ 
is the smallest ordinal number such that $K_1^{(\beta)}$ is finite. Furthermore, since 
$|K_1^{(\beta)}|=p$, we obtain that $|K_2^{(\beta)}|=|K_1^{(\beta)}|=p$. With 
this, we conclude that $\mathcal{CB}(K_2)=(\beta,p)=\mathcal{CB}(K_1)$.
\end{proof}

\begin{Proposition}\label{prop:3.02}
Let $(E,d)$ be a metric space and let $K_1,K_2\in \mathcal{K}_E$. If
$\mathcal{CB}(K_1)=\mathcal{CB}(K_2)$, then $K_1\sim K_2$.
\end{Proposition}
\begin{proof}
Let $\mathcal{CB}(K_1)=\mathcal{CB}(K_2)=(\alpha,p)$, for some ordinal number $\alpha$
and some $p\in\omega$. 
\begin{itemize}
\item If $\alpha=0$, we have that $K_1$ and $K_2$ are both finite sets with $p$
      elements, therefore $K_1\sim K_2$.

\item If $\alpha>0$ and $p\in \omega$, by Theorem~\ref{Theorem}, we have that 
      $K_1\sim \omega^\alpha\cdot p +1$ and $K_2\sim \omega^\alpha\cdot p +1$, 
      thus $K_1\sim K_2$.\qedhere
\end{itemize}
\end{proof}

Propositions~\ref{prop:3.01} and~\ref{prop:3.02} imply that for a metric 
space $(E,d)$, the partition of $\mathcal{K}_E$ is fully characterized by the 
Cantor-Bendixson characteristic.

\begin{Theorem}\label{teo:P4}
Let $(E,d)$ be a metric space. The cardinality of $\Kp_E$ is less than or equal to
$\aleph_1$.
\end{Theorem}
\begin{proof}
We define the function
\begin{equation*}
    \funcion{\widetilde{\mathcal{CB}}}{\Kp_E}{\omega_1\times\omega}
    {[K]}{\widetilde{\mathcal{CB}}([K])=\mathcal{CB}(K).}
\end{equation*}
By Theorem~\ref{T1} and Proposition~\ref{prop:3.01}, we see that function 
$\widetilde{\mathcal{CB}}$ is well-defined.  Moreover, Proposition~\ref{prop:3.02} 
shows that $\widetilde{\mathcal{CB}}$ is an injective function. Thus,
\begin{equation*}
    |\Kp_E| \leq  |\omega_1\times\omega|=|\omega_1|=:\aleph_1.\qedhere
\end{equation*}
\end{proof}

In general, we cannot strengthen the last result. To show this, we give 
the following three propositions.

\begin{Proposition}\label{prop:3.03}
For all $n\in\omega$, there exists a metric space $(E_n,d_n)$ such that 
$|\Kp_{E_n}|=n$.
\end{Proposition}
\begin{proof}
Let $n\in\omega$.  We now consider the space $(n,\rho_n)$, where $\rho_n$ is 
the discrete metric on the set $n:=\{0, \ldots, n-1\}$. Since $n$  
is a finite set, we have that every subset of $n$ is compact, i.e.,
\begin{equation*}
	\mathcal{K}_n=\mathcal{P}(n).
\end{equation*}
Moreover, for all $K\in\mathcal{K}_n$, we see that $K^{(0)}=K$ is a finite set. 
Therefore, $\mathcal{CB}(K)=(0,|K|)$. Thus, 
\begin{equation*}
    \widetilde{\mathcal{CB}}(\Kp_n)\subseteq \{0\}\times n.
\end{equation*}
On the other hand, for all $k \in \{0, \ldots n-1\}$, there exists 
$F\subseteq n$ such that $|F|=k$, i.e., $\mathcal{CB}(F)=(0,k)$.  Thus,
\begin{equation*}
    \widetilde{\mathcal{CB}}(\Kp_n) = \{0\}\times n.
\end{equation*}
Hence,
\begin{equation*}
	|\Kp_n|= |\widetilde{\mathcal{CB}}(\Kp_n)| = n.\qedhere
\end{equation*}
\end{proof}
\begin{Proposition}\label{prop:3.04}
There exists a metric space $(E,d_E)$ such that $|\Kp_E|=\aleph_0$.
\end{Proposition}
\begin{proof}
We consider the metric space $(\omega, \rho_\omega)$.  Since $\rho_\omega$ is 
the discrete metric on the set $\omega$, we see that a subset of $\omega$ is compact 
if and only if it is a finite set. Then, for all $K\in\mathcal{K}_\omega$,  
$K^{(0)}=K$ is a finite set.  Hence, for all $K\in\mathcal{K}_\omega$, 
$\mathcal{CB}(K)=(0,|K|)$. Thus, 
$\widetilde{\mathcal{CB}}(\Kp_{\omega}) \subseteq \{0\} \times \omega$.  On 
the other hand, since for all $k\in \omega$, there exists $K \subseteq \omega$ 
such that $|K|=k$, it follows that
$\widetilde{\mathcal{CB}}(\Kp_{\omega}) = \{0\} \times \omega$.  Therefore, 
\begin{equation*}
	|\Kp_\omega|= |\widetilde{\mathcal{CB}}(\Kp_{\omega})| = \aleph_0.\qedhere
\end{equation*}
\end{proof}

\begin{Proposition}\label{prop:3.05}
There exists a metric space $(F,d_F)$ such that $|\Kp_F|=\aleph_1$.
\end{Proposition}
\begin{proof}
We take the metric space $(\mathbb{R},d)$, where $d$ is the usual metric on 
the set $\mathbb{R}$. By Theorem 3.4 in \cite{Alvarez-Merino2016}, we obtain that
\begin{equation*}
	|\Kp_\mathbb{R}|=\aleph_1.\qedhere
\end{equation*}
\end{proof}

Finally, it is worth mentioning that the last two results do not depend on the 
cardinality of the underlying metric spaces considered there, as it can be seen 
in the next two propositions.
\begin{Proposition}\label{prop:3.06}
There exists a countable metric space $(G,d_G)$ such that $|\Kp_G|=\aleph_1$.
\end{Proposition}
\begin{proof}
Proceeding in a similar way as in the proof of Theorem 2.1 in 
\cite{Alvarez-Merino2016} and considering the density of the rational numbers, 
$\mathbb{Q}$, in $\mathbb{R}$, we can see that for all countable ordinal number 
$\alpha \in \omega_1$, and for all $a, b \in \mathbb{Q}$ such that $a < b$, there 
exists a set $K \in \mathcal{K}_{\mathbb{Q}}$ such that $K \subseteq (a, b]$ and 
$K^{(\alpha)} = \{b\}$.  By using this last statement, we can prove an analogous 
result to Corollary 2.1 in~\cite{Alvarez-Merino2016}, more precisely, we have that 
for all countable ordinal number $\alpha \in \omega_1$, and for all $p\in \omega$, 
there is a set $K \in \mathcal{K}_{\mathbb{Q}}$ such that $|K^{(\alpha)}|=p$.  Then,
\begin{equation*}
	\widetilde{\mathcal{CB}}(\Kp_\mathbb{Q})=
	\big(\omega_1\times(\omega\smallsetminus\{0\})\big)\cup(0,0).
\end{equation*}
Hence,
\begin{equation*}
	|\Kp_\mathbb{Q}|=|\widetilde{\mathcal{CB}}(\Kp_\mathbb{Q})|
	= |\omega_1\times\omega| = |\omega_1| =: \aleph_1.
	\qedhere
\end{equation*}

\end{proof}

\begin{Proposition}\label{prop:3.07}
There exists an uncountable metric space $(H,d_H)$ such that $|\Kp_H|=\aleph_0$.
\end{Proposition}
\begin{proof}
We take the uncountable metric space $(\mathbb{R},\rho_\mathbb{R})$, where 
$\rho_\mathbb{R}$ is the discrete metric on the real line. Proceeding in a 
similar fashion as in the proof of Proposition~\ref{prop:3.04}, we obtain
\begin{equation*}
	|\Kp_{\mathbb{R}}|=\aleph_0.\qedhere
\end{equation*}
\end{proof}

\nocite{*}

\bibliographystyle{siam}
\bibliography{CBFunction-arXiv1}

\end{document}